\documentclass[12pt]{amsart}
\usepackage{amscd}
\usepackage{amsaddr}
\usepackage{amssymb}
\usepackage{a4wide}
\usepackage{amstext}
\usepackage{amsthm}
\usepackage{mathrsfs}
\usepackage{color}

\renewcommand{\phi}{\varphi}

\newcommand{\N}{\mathbb{N}}

\newcommand{\D}{\mathbb{D}}

\newcommand{\T}{\mathbb{T}}

\newtheorem{Thm}{Theorem}[section]
\newtheorem{theorem}[Thm]{Theorem}
\newtheorem{lemma}[Thm]{Lemma}

\newtheorem{remark}[Thm]{Remark}

\newcommand{\hb}{\mathscr{H}_\beta}
\newcommand{\om}{\omega}

\begin{document}
\sloppy

\title[]{Representing systems of reproducing kernels \\ in spaces of analytic functions}

\author{Anton Baranov$^*$, Timur Batenev}

\address{Department of Mathematics and Mechanics, St.~Petersburg State University, St.~Petersburg, 198504, Russia}

\address{$^*$Corresponding author. E-mail:  anton.d.baranov@gmail.com}






\thanks{The work is supported by Ministry of Science and Higher Education of the Russian Federation (agreement No 075-15-2021-602) and by Theoretical Physics and Mathematics Advancement Foundation ``BASIS''}

\begin{abstract} {We give an elementary construction of representing systems of the Cauchy kernels
in the Hardy spaces $H^p$, $1 \le p <\infty$, as well as of
representing systems of reproducing kernels in weighted Hardy spaces.} 
\end{abstract}

\maketitle


\section{Introduction and main results}
\label{section1}

A system $\{x_n\}_{n\ge 1}$ in a separable infinite-dimensional Banach space $X$
is said to be a {\it representing system for $X$} if, for every
element $x\in X$, there exists a sequence of complex numbers $\{c_n\}_{n\ge 1}$  such that
$$
x=\sum_{n\ge 1} c_n x_n,
$$
where the series converges in the norm of $X$. In contrast to the (probably better known) notion of the Schauder basis
we do not require that the coefficients in this representation are unique. 

Representing systems were much studied both in the general functional analysis context and 
for some specific systems in functional spaces. E.g., there exists a vast literature dealing with representing systems of 
exponentials in various Frech\'et spaces of analytic functions (see surveys \cite{is, kor}). 
However, it seems that representing systems 
of reproducing kernels in classical spaces of analytic functions in the disk did not attract much attention until recently.
\medskip

\subsection{Classical Hardy spaces}
In \cite{fkl} E.~Fricain, L.\,H.~Khoi and P.~Lef\`evre addressed the existence problem 
for the representing and absolutely representing systems of
reproducing kernels in reproducing kernel Hilbert spaces and showed that many classical spaces do not possess
absolutely representing systems of reproducing kernels. The question about existence of
representing systems remained open. In particular, in \cite{fkl}, the authors asked the following 
\medskip
\\
{\bf Question.} {\it Do there exist
sequences $\Lambda = \{\lambda_n\}_{n\ge 1} \subset \D$ such that the system $\mathcal{K}(\Lambda) = 
\{k_{\lambda_n}\}_{n\ge 1}$,
where 
$$
k_\lambda(z) = \frac{1}{1-\bar \lambda z}
$$
is the Cauchy \textup(or Szeg\"o\textup) kernel at $\lambda$, is representing for the Hardy space $H^2$ in the unit disk $\D$? }

The positive answer to this question was given by K.\,S.~Speranskii and P.\,A.~Terekhin \cite{st1}. Namely, it was shown in \cite{st1} that
for the sequence 
$$
\Lambda = \Big\{ \lambda_{k,j} = \Big(1-\frac{1}{k} \Big) e^{\frac{2\pi i j}{k}}:  k\ge 1, j=0, 1, \dots k-1\}
$$
the system $\mathcal{K}(\Lambda)$ is representing for $H^2$. The sequence $\Lambda$ is assumed to have 
the standard alphabetical order: $\lambda_{1, 0}, \lambda_{2,0}, \lambda_{2,1}, \lambda_{3,0}, \dots$. In what follows
we always assume that sequences with double (or triple) index will be ordered in this way.
In \cite{st2} Speranskii and Terekhin extended their result to a more general class of sequences. 
Let $n_k \in \mathbb{N}$ and  $r_k\to 1-$, $k\to\infty$. 
Define the sequence $\Lambda$ by
\begin{equation}
\label{loop}
\Lambda = \Big\{ \lambda_{k,j} = r_k e^{\frac{2\pi i j}{n_k}}:  k\ge 1, j=0, 1, \dots n_k-1\}.
\end{equation}
As shown in \cite{st2}, if there exist positive constants $A$ and $B$ such that 
$A \le n_k(1-r_k) \le B$ for all $k$, then $\mathcal{K}(\Lambda)$ is a representing system for $H^2$. 

The proofs in \cite{st1, st2} are based on interesting abstract functional analysis methods from the papers \cite{ter1, ter2} 
which relate representing systems with coefficients from a given function space to a certain generalized notion of a frame
(see Section \ref{hardf} for details). 

The goal of the present work is to give {\it a very simple elementary construction} of a representing 
system of the Cauchy kernels, which does not make use of functional analysis. 
The idea is to use a discretization of the Cauchy formula. This method applies
to all Hardy spaces $H^p$, $1<p<\infty$, but does not cover all systems of the form \eqref{loop}
with $A \le n_k(1-r_k) \le B$:  it is required that the constant $A$ is sufficiently large. 
However, an application of the frame theory method by 
Speranskii and Terekhin allows us to prove the result for any $A>0$.

\begin{theorem} 
\label{hardy}
If $\Lambda$ is given by \eqref{loop} and there exists $M>0$ such that $n_k (1-r_k) \ge M$ for any $k$,
then $\mathcal{K}(\Lambda)$ is a representing system for $H^p$ for any $p \in (1, \infty)$.
\end{theorem}

We will give two proofs of Theorem \ref{hardy}. The first one is completely elementary
and constructive, but applies only to the case $ M> \pi$, while
the second one works for any $M>0$. Both of these proofs do
not extend to the case $p=1$. The main obstacle for the first method is in the fact 
that the Cauchy transform is not bounded
in $L^1$. However, one can construct representing systems of the Cauchy kernels in $H^1$
if we take the points uniformly distributed on the circle $\{|z| = 1- 1/n_k\}$ with certain logarithmic multiplicities.
For a precise formulation see Theorem \ref{hardy1}.

It is obvious that there are no representing systems of the Cauchy kernels in $H^\infty$, since the uniform limit
of their finite linear combinations belongs to the disk-algebra $A(\D)$ (the space of all functions continuous
in $\overline{\D}$ and analytic in $\D$ equipped with the usual sup-norm). However, the systems of the
Cauchy kernels from Theorem \ref{hardy1} are representing also for the disk algebra $A(\D)$.
\medskip


\subsection{Weighted Hardy spaces}
Our second result concerns the class of weighted Hardy spaces $\hb$ in the disk. Let
the sequence $\beta = \{\beta_n\}_{n=0}^{\infty}$, $\beta_n > 0$, 
satisfy
\begin{equation}
\label{ls}
\limsup_{n\to \infty}\, (1/\beta_n)^{1/n} \le 1, \qquad 
\limsup_{n\to \infty} \beta_n^{1/n} \le 1.
\end{equation}
Consider the set of analytic functions 
$$
\mathscr{H}_\beta= \bigg\{f(z)=\sum\limits_{n=0}^{\infty}a_n z^n: \ 
\sum\limits_{n=0}^{\infty} |a_n|^2\beta_n < +\infty \bigg\}.
$$ 
It follows from \eqref{ls} that $\hb$ consists of functions analytic in the unit disk $\D$ and contains functions 
which are not analytic in any larger disk. 
It is clear that $\hb$ is a reproducing kernel Hilbert space with respect to the norm
$\|f\|_\beta^2= \sum\limits_{n=0}^{\infty} |a_n|^2\beta_n$ and its kernel 
at the point $\lambda\in\D$ is given by  
$$
K^\beta_\lambda(z)= \sum\limits_{n=0}^{\infty}\frac{\overline{\lambda}^n}{\beta_n}z^n.
$$

Weighted Hardy spaces $\hb$ include most of the classical spaces of analytic functions in the unit disk:
the Hardy space $H^2$ ($\beta_n \equiv 1$), Bergman spaces $A^2_\alpha$ with the weight $(\alpha+1) (1-|z|^2)^\alpha$, 
$\alpha>-1$ ($\beta_n = \frac{n!\Gamma(\alpha+2)}{\Gamma(n+\alpha+2)}$), 
the Dirichlet space ($\beta_n = n+1$).

Recall that a sequence $\{x_n\}$ is said to be a {\it frame} in a Hilbert space $H$ if there exist constants $A,B>0$
such that $A\|x\|^2 \le \sum_n |(x, x_n)|^2 \le B\|x\|^2$ for any $x\in H$;
if one has only the above estimate $\sum_n |(x, x_n)|^2 \le B\|x\|^2$, then $\{x_n\}$ is said to be a
{\it Bessel sequence}. Any frame is, in particular, a representing system.

It is well known that in Bergman spaces $A^2_\alpha$ there exist frames of normalized reproducing kernels; 
their complete description was given by K.~Seip \cite{seip} (for general weighted Bergman spaces see 
\cite{bdk, seip1}). Therefore, the existence of representing 
sequences of reproducing kernels (but not their description) in the Bergman space setting is trivial. 
On the other hand, $H^2$ has no frames of normalized Cauchy kernels (and even complete Bessel sequences). 
Indeed, for any Bessel sequence $\{ k_{z_n}/\|k_{z_n}\|_{H^2} \}$ one has $\sum_n (1-|z_n|^2) <\infty$ 
(simply applying the inequality with $f\equiv 1$), whence $\{z_n\}$ 
is a Blaschke ($=$nonuniqueness) sequence. 

More generally, if $\inf_n \beta_n =\delta >0$, then $\hb$ has no frames of normalized reproducing kernels. 
Indeed, if $\{K^\beta_{z_n}/\|K^\beta_{z_n}\|_\beta\}$ is a frame, then $\sum_n \|K^\beta_{z_n}\|^{-2}_\beta <\infty$.
Let $B_N$ be the Blaschke product with the zeros $z_1, \dots, z_N$. Then 
$$
\sum_{n>N} |B_N(z_n)|^2 \|K^\beta_{z_n}\|^{-2}_\beta \le \sum_{n>N} \|K^\beta_{z_n}\|^{-2}_\beta \to 0
$$
as $N\to \infty$. On the other hand, since $\beta_n\ge \delta$, 
we have $\|B_N\|^2_\beta \ge \delta \|B_N\|^2_{H^2} =\delta$, and we come to a contradiction 
with the frame inequality.

Thus, weighted Hardy spaces which are smaller than $H^2$ (e.g., the Dirichlet space) possess no 
frames of normalized reproducing kernels and the problem about existence of representing systems
of reproducing kernels becomes nontrivial. We give a positive  answer to this  question.

\begin{theorem}
\label{wei}
For any sequence $\beta$ satisfying \eqref{ls} 
in the space $\hb$ there exist representing systems of reproducing kernels.
\end{theorem}
\bigskip


\section{Simple proof of Theorem \ref{hardy}}
\label{hard}
Recall that the Hardy space $H^p$, $1 \le p<\infty$, consists of all functions $f$ analytic in $\D$
and such that
$$
  \|f\|^p_{H^p}=\sup\limits_{0<r<1} 
\int_\T |f(r \zeta)|^p\,dm(\zeta) <\infty.
$$ 		
Here $m$ denotes the normalized Lebesgue measure on the unit circle $\T$. 
Since $H^p$ is a closed subspace of $L^p(\T)$ in what follows we sometimes denote the norm in $H^p$ and $L^p$
by $\|\cdot\|_p$.

For any $f\in H^p$ one has
\begin{equation}
\label{repr}
f(z) = \int_\T f(\zeta)\overline{k_z(\zeta)} dm(\zeta)= \int_\T \frac{f(\zeta)}{1-\bar \zeta z}dm(\zeta), \qquad z\in \D.
\end{equation}
In particular, $k_z$ is the reproducing kernel of $H^2$ at the point $z\in\D$
and the Cauchy transform
$$
(\mathcal{C} g) (z) = \int_\T \frac{g(\zeta)}{1-\bar \zeta z}dm(\zeta), \qquad z\in \D, 
$$
is the orthogonal projection of a function $g\in L^2(\T)$ to $H^2$. The same is true for any $p\in (1, \infty)$: 
there exists $C_p>0$ such that for any $g\in L^p(\T)$ one has $\mathcal{C} g \in H^p$ and $\|\mathcal{C} g\|_p 
\le C_p \|g\|_p$.

The idea of the proof of Theorem \ref{hardy} is to replace the integral \eqref{repr} by a certain ``discretization''. 
\medskip
\\
{\bf Proof of Theorem \ref{hardy}.}
When we approximate a given function $f\in H^p$, the points (or rather layers) of $\Lambda$ given  by \eqref{loop} 
will be defined inductively. We first explain one step of induction. 
Let $f\in H^p$ be given. Put $f_r(z) = f(rz)$. It is well known that $\|f -f_r\|_p \to 0$, $r\to 1-$.
Therefore, for any positive $\delta$ (to be specified later) we can 
choose $r_k$ such that $\|f- f_{r_k}\|_p \le \delta \|f\|_p$. 

Let $I_j = I_{k,j}$, $0\le j\le n_k-1$, be the arcs of $\T$ defined as
\begin{equation}
\label{arc}
I_j = I_{k,j} = \Big[\exp \Big(\frac{(2j-1) \pi i }{n_k}\Big), \exp \Big(\frac{ (2j+1)\pi i}{n_k}\Big)\Big].
\end{equation}
and let $\zeta_j =\zeta_{k,j}= \exp \big( \frac{2\pi i j}{n_k} \big)$.  
At this step the index $k$ is fixed, thus, we omit it and write simply $I_j$, $\zeta_j$. Then
$$
f(r_k z) =\int_\T \frac{f(\zeta)}{1-r_k \bar \zeta z}dm(\zeta) = \sum_{j=0}^{n_k-1}
\int_{I_j} \frac{f(\zeta)}{1-r_k \bar \zeta z}dm(\zeta). 
$$
Now it is natural to approximate $f_{r_k}(z) = f(r_k z)$ by
$$
S(z) = \sum_{j=0}^{n_k-1}
\frac{1}{1- r_k \bar \zeta_j z}
\int_{I_j} f(\zeta) dm(\zeta).
$$

Let us show that if $n_k(1-r_k) \ge M>\pi$, then there
exists a numeric constant $\gamma \in (0,1)$ such that for all sufficiently large $k$ one has
\begin{equation}
\label{dod}
\| f_{r_k}  - S \|_{H^p} \le \gamma \|f\|_p.
\end{equation}
Hence, $\| f- S\|_{H^p} \le (\gamma+\delta) \|f\|_p$ 
and we need to choose $\delta>0$ so that $\gamma+\delta <1$. We have
$$
f(r_k z) - S(z) = \sum_{j=0}^{n_k-1}
\int_{I_j} \frac{r_k z(\bar \zeta - \bar \zeta_j)}{(1-r_k \bar \zeta_j z )(1 - r_k \bar \zeta z)}  f(\zeta) dm(\zeta).
$$
Note that for any $\zeta\in I_j$ we have $|\zeta-\zeta_j| \le \pi n_k^{-1} <1-r_k$ and
therefore $|1-r_k \bar \zeta z| \le 2 |1 - r_k \bar \zeta_j z|$. Thus, 
\begin{equation}
\label{dodr}
|f(r_k z) - S(z)| \le \frac{2\pi}{n_k} \sum_{j=0}^{n_k-1}
\int_{I_j} \frac{|f(\zeta)|}{|1 - r_k \bar \zeta z|^2} dm(\zeta) = 
\frac{2\pi}{n_k}\int_{\T} \frac{|f(\zeta)|}{|1 - r_k \bar \zeta z|^2} dm(\zeta).
\end{equation}
By the H\"older inequality ($1/p +1/q = 1$), we have
$$
\|f_{r_k} - S\|_{H^p}^p \le 
\frac{(2\pi)^p}{n_k^p} \int_\T \bigg( \int_{\T} \frac{|f(\zeta)|^p}{|1 - r_k \bar \zeta z|^2} dm(\zeta)\bigg) 
\bigg( \int_{\T} \frac{dm(\zeta)}{|1 - r_k \bar \zeta z|^2} \bigg)^{p/q} dm(z).
$$
Since $\int_{\T} |1 - r_k u|^{-2} dm(u) =(1-r_k^2)^{-1}$, 
we conclude that
$$
\|f_{r_k} - S\|_{H^p} \le 2\pi \frac{\|f\|_{H^p}}{n_k(1-r_k^2)} \le \frac{2\pi}{M(1+r_k)} \|f\|_{H^p}. 
$$
Note that $r_k$ can be chosen as close to $1$ as we wish. Hence, since $M>\pi$, we have 
$\|f_{r_k} - S\|_{H^p} \le \gamma \|f\|_{H^p}$ for some absolute numeric constant  $\gamma\in (0,1)$. 
Since $\delta$ also can be chosen as small as we wish, we get $\|f - S\|_{H^p} \le \gamma \|f\|_{H^p}$
with another numeric constant  $\gamma\in (0,1)$ and for all sufficiently large $k$. 

Also, note that there exists a constant $B_p>0$ (depending only on $p$)  such that 
for any $0\le n\le n_k-1$ one has
\begin{equation}
\label{dod1}
\bigg\|\sum_{j=0}^{n} \frac{1}{1- r_k \bar \zeta_j z}
\int_{I_j} f(\zeta) dm(\zeta)\bigg\|_{H^p}
\le B_p \|f\|_{H^p}.
\end{equation}
Indeed, above we already showed that, for any $n\le n_k-1$, 
$$ 
\bigg\|\sum_{j=0}^{n}\frac{1}{1- r_k \bar \zeta_j z}
\int_{I_j} f(\zeta) dm(\zeta)  - 
\int_{\cup_{j=0}^n I_j} \frac{f(\zeta)}{1- r_k \bar \zeta z}dm(\zeta)
\bigg\|_{H^p} \le \gamma \|f\|_{H^p},
$$
while for the second term we
use the boundedness of the Cauchy transform in $L^p$, $1<p<\infty$:
$$ 
\bigg\| \int_{\cup_{j=0}^n I_j} \frac{f(\zeta)}{1- r_k \bar \zeta z}dm(\zeta)
\bigg\|_{H^p} \le C_p \|f\|_{H^p}.
$$

Now, everything is ready to complete the proof. We start with an arbitrary function $f\in H^p$ and choose $r_{k_1}$ as
described above to obtain a function  
$$
f_1(z) = f(z) -  \sum_{j=0}^{n_{k_1}-1}
\frac{1}{1- r_{k_1} \bar \zeta_{k_1, j} z}
\int_{I_{k_1, j}} f(\zeta) dm(\zeta)
$$
with $\|f_1\|_{H^p} \le \gamma \|f\|_{H^p}$, where $\gamma\in (0,1)$.

Next, we apply the same procedure to $f_1$ and find $r_{k_2}$ such that 
$\|f_2\|_{H^p} \le \gamma \|f_1\|_{H^p}$, where
$$
f_2(z) = f_1(z) -  \sum_{j=0}^{n_{k_2}-1}
\frac{1}{1- r_{k_2} \bar \zeta_{k_2, j} z}
\int_{I_{k_2, j}} f_1(\zeta) dm(\zeta).
$$
Proceeding in this way,  we obtain a sequence $k_l$ and a sequence of coefficients $c_{l,j} =
\int_{I_{k_l, j}} f_{l-1}(\zeta) dm(\zeta)$ such that
$$
\|f_N\|_{H^p} = \bigg\|f- \sum_{l=1}^N \sum_{j=0}^{n_{k_l}-1}
\frac{c_{l,j}}{1- r_{k_l} \bar \zeta_{k_l,j} z}\bigg\|_{H^p} \le \gamma^N \|f\|_{H^p}. 
$$

It remains to show that the series
$$
\sum_{l=1}^\infty \sum_{j=0}^{n_{k_l}-1}
\frac{c_{l,j}}{1- r_{k_l} \bar \zeta_{k_l,j} z}
$$
converges to $f$ in the norm of $H^p$. Indeed, for any $N\in \mathbb{N}$ and $0\le n\le  n_{k_{N+1}} -1$
we have
$$
\begin{aligned}
\bigg\|f & - \sum_{l=1}^N \sum_{j=0}^{n_{k_l}-1}
\frac{c_{l,j}}{1- r_{k_l} \bar \zeta_{k_l,j} z}  - 
\sum_{j=0}^{n} \frac{c_{N+1,j}}{1- r_{k_{N+1}} \bar \zeta_{k_{N+1},j} z}\bigg\|_{H^p} \\
& \le \bigg\|f- \sum_{l=1}^N \sum_{j=0}^{n_{k_l}-1}
\frac{c_{l,j}}{1- r_{k_l} \bar \zeta_{k_l,j} z} \bigg\|_{H^p} +
\bigg\| \sum_{j=0}^{n} \frac{c_{N+1,j}}{1- r_{k_{N+1}} \bar \zeta_{k_{N+1},j} z}\bigg\|_{H^p} \\
& \le  (1+B_p) \|f_N\|_{H^p} \le (1+B_p)  \gamma^N \|f\|_{H^p} \to 0, \qquad
N\to\infty.
\end{aligned}
$$
Here we used inequality \eqref{dod1}. The proof is completed. 
\qed

\begin{remark}
{\rm The same proof shows that we need not take $\zeta_j$  as the centers of the arc $I_j$ and can choose them
randomly in $I_j$. Repeating the arguments one immediately obtains that there exists $M>0$ such that if $n_k(1-r_k) \ge M$,
then the sequence $\Lambda = \{r_k \zeta_{k,j}: \zeta_{k,j} \in I_{k,j}, \ k\in\mathbb{N}, \ 0\le j<n_k-1\}$ 
generates a   system of Cauchy kernels which is representing in any $H^p$, $1<p<\infty$. }
\end{remark}
\bigskip


\section{Frame theory proof of Theorem \ref{hardy}}
\label{hardf}

In this section we give a proof of Theorem \ref{hardy} based on the general methods due to P.\,A.~Terekhin.

Let $F$ be a Banach space, $F^*$ be its dual
and let $X$ be a Banach space of sequences where the canonical basis vectors $e_n = (\delta_{k,n})_k$ form a 
Schauder basis. Then its dual $X^*$ also can be identified with a space of sequences. A system ${f_n}$ in $F$
is said to be  a {\it frame with respect to the space $X$} if for any $\varphi\in F^*$ one has
$$
A\|\varphi\|_{F^*} \le \|(\varphi(f_n))_n\|_{X^*} \le B \|\varphi\|_{F^*}
$$
for some $A, B >0$ (here the sequence $(\varphi(f_n))_n$ is considered as an element of $X^*$.

We will use the following result of P.\,A.~Terekhin \cite[Theorem 4]{ter1}: {\it if 
${f_n}$ is frame for $F$ with respect to $X$, then ${f_n}$ is a representing system in $F$
and any $f\in F$ can be represented as the sum of the series $f=\sum_n c_n f_n$ with $(c_n) \in X$.}

Now let $n_k\in \mathbb{N}$, $n_k\to\infty$, and $r_k \in (0,1)$ be such that $M \le n_k(1-r_k) \le \tilde M$ 
for some $M, \tilde M >0$  and for all $k$. Let 
$$
\Lambda = \Big\{ \lambda_{k,j} = |\lambda_{k,j}| e^{i\alpha_{k, j}}:  k\ge 1, j=0, 1, \dots n_k-1\},
$$
where, for some $a, b, c, d>0$, all $k$ and $j=0, 1, \dots n_k-1$,
\begin{equation}
\label{cont}
a(1-r_k) \le 1- |\lambda_{k,j}| \le b(1-r_k), \qquad 
\frac{c}{n_k} \le \alpha_{k, j+1} - \alpha_{k, j} \le \frac{d}{n_k};
\end{equation}
here, by definition, $\alpha_{k, n_k} = \alpha_{k, 0} +2\pi$.

Let $F=H^p$, where $1<p<\infty$ and $1/p+1/q=1$. Then $F^* = L^q/z H^q \cong \overline{H^q}$ with equivalence of norms, i.e., 
for any functional $\varphi \in (H^p)^*$ there exists $g\in H^q$ such that $\varphi(f) = \int_{\T} f\bar g\, dm$
and $\|\varphi\| \asymp \|g\|_{H^q}$ with constants depending on $p$ only. 
Define the sequence spaces $X = \Big(\bigoplus\limits_{k=1}^{\infty} \ell^p_{n_k}\Big)_{\ell^1}$
and $X^*=\Big(\bigoplus\limits_{k=1}^{\infty} \ell^q_{n_k}\Big)_{\ell^{\infty}}$ with the norms
$$
\|(c_{k,j})\|_X = \sum_{k=1}^\infty \Bigg(\sum_{j=0}^{n_k-1} |c_{k,j}|^p\Bigg)^{1/p}, 
\qquad
\|(c_{k,j})\|_{X^*} = \sup_k \Bigg(\sum_{j=0}^{n_k-1} |c_{k,j}|^q\Bigg)^{1/q}.
$$

Recall that $\|k_\lambda\|_{H^p} \asymp (1-|\lambda|)^{-1/q}$ and consider the system of (almost) normalized kernels
$\{(1-r_k)^{1/q} k_{\lambda_{k,j}}\}$. Note that for any $g\in H^q$ one has
$\int_\T g(z) \overline{k_\lambda(z)} dm(z)=g(\lambda)$. Therefore, to verify that 
$\{(1-r_k)^{1/q} k_{\lambda_{k,j}}\}$ is a frame for $H^p $  with respect to $X$ we need to show that 
\begin{equation}
\label{cont1}
A\|g\|_{H^q} \le \sup_k \bigg( \sum_{j=0}^{n_k-1} (1-r_k) |g(\lambda_{k,j})|^q \bigg)^{1/q} \le B
\|g\|_{H^q} 
\end{equation}
for some $A,B>0$ and any $g\in H^q$. 

The above estimate follows from the basic facts about the Carleson embeddings of the Hardy spaces
(see, e.g., \cite{dur}). 
Recall that a Borel  measure $\mu $ in $\D$ is said to be a Carleson measure if there exists $C(\mu)>0$ 
such that  for all $\zeta = e^{i\theta}\in\T$ and $h\in (0,1]$
$$
\mu(S(\zeta, h)) \le C(\mu) h,
$$
where $S(\zeta, h) = \{z=re^{i\phi}: 1-h\le r<1, |\phi-\theta|<h\}$ is a Carleson ``square''.
If $\mu$ is a Carleson measure, then, for any $f\in H^p$,
$$
\int_{\D} |f(z)|^p d\mu(z) \le A C(\mu) \|f\|_{H^p}^p,
$$
where $A$ is some absolute numeric constant. 

Consider the measures $\mu_k = \sum_{j=0}^{n_k-1} (1-r_k) \delta_{\lambda_{k,j}}$. From conditions \eqref{cont} 
it follows immediately that $C(\mu_k) \le C$ for some constant $C$ depending only on $M,\tilde M, a,b,c,d$, but not on $k$.
This proves the right-hand side estimate in \eqref{cont1}. 

Note that, in view of the already established upper bound, it is sufficient to prove 
the lower estimate in \eqref{cont1} for a dense subset of $H^q$, e.g., for functions continuous in $\overline{\D}$.
Consider the arcs $I_j = [e^{i\alpha_{k, j}}, e^{i\alpha_{k, j+1}}]$,
$j=0, \dots, n_k-1$, and note that $|I_j| \asymp 1-r_k$. Then, by the H\"older inequality,
$$
\begin{aligned}
\bigg( \sum_{j=0}^{n_k-1} & |I_j|\cdot  |g(\lambda_{k,j})|^q \bigg)^{1/q}   =
\bigg( \sum_{j=0}^{n_k-1} \int_{I_j} |g(\lambda_{k,j})|^q dm(\zeta) \bigg)^{1/q}
\\ & \ge
\bigg( \sum_{j=0}^{n_k-1} \int_{I_j} |g(r_k \zeta)|^q dm(\zeta) \bigg)^{1/q} -
\bigg( \sum_{j=0}^{n_k-1} \int_{I_j} |g(r_k \zeta) - g(\lambda_{k,j})|^q dm(\zeta) \bigg)^{1/q}.
\end{aligned}
$$
It is clear that the first term tends to $\|g\|_{H^q}$ as $k\to \infty$, while the second
tends to 0, since, by our assumption, $g$ is uniformly continuous in $\overline{\D}$.
Thus, the left-hand estimate in \eqref{cont1} is established, which completes the proof.
\qed
\bigskip


\section{Representing systems in $H^1$ and in $A(\D)$} 

It is clear from the proof of Theorem 1.1 (see estimate \eqref{dodr}) that the function $S$ well approximates
$f$ even in the cases when $p = 1$ or $f\in A(\D)$. Note that, in contrast to $H^\infty$ case, 
$\|f_r - f \|_{A(\D)} \to 0$, $r\to 1-$, if $f\in A(\D)$. The only problem arises when 
we need to estimate the norm of a discretization of the integral over an arc of the circle. Since
the Cauchy transform is unbounded in $L^1$ and in $L^\infty$, these norms can be large. 

We can construct representing systems of Cauchy kernels in $H^1$ or in $A(\D)$ by considering more
dense sets distributed over a circle with a certain ``multiplicity''.
Let $R_k \in(0,1)$, $N_k, M_k \in \N$. 
For any $j$, $1\le j \le N_k$, consider the open arc 
$I_{k,j} = (\exp(\frac{(2 j -1)\pi i}{N_k}), \exp(\frac{(2 j +1)\pi i}{N_k})) \subset \T$ and choose
$M_k$ distinct points $\zeta_{k, l, j} \in I_{k,j}$, $l=1, \dots, M_k$. Define the set 
\begin{equation}
\label{bor}
\Lambda = \{w_{k, l, j} = R_k \zeta_{k, l, j} :\ k\in\N, \ 1\le l \le M_k, \ 1\le j\le N_k\}.
\end{equation}
The set $\Lambda$ is assumed to be ordered alphabetically. 
We prefer to make the points in $\Lambda$ distinct even if 
the definition of a representing system does not exclude repeating vectors. 

\begin{theorem}
\label{hardy1}
Let $\Lambda$ be given by \eqref{bor} with $R_k \to 1-$, $k\to\infty$. Then there exists
a numeric constant $M>0$ such that if $N_k(1-R_k) \ge M$ and
$$
\log \frac{1}{1-R_k} = O(M_k),
$$
then $\mathcal{K}(\Lambda) =\{k_\lambda\}_{\lambda\in \Lambda}$ is a representing system in $H^1$
and in $A(\D)$.
\end{theorem}

In what follows we write $X\lesssim Y$ if 
there is a constant $C>0$ such that $X\le C Y$ for all admissible values of parameters. 

\begin{proof} We start with a trivial formula
$$
f(R_k z) = \frac{1}{M_k} \sum_{l=1}^{M_k} \int_\T \frac{f(\zeta)}{1-R_k \bar \zeta z} dm(\zeta)
$$
and its discretization
$$
S_k(z) =  \frac{1}{M_k} \sum_{l=1}^{M_k} \sum_{j=1}^{N_k} \bigg(\int_{I_{k,j}} f(\zeta) dm(\zeta) \bigg)
\frac{1}{1-R_k \bar \zeta_{k,l,j} z}.
$$

First we consider the case of the space $H^1$.
Repeating the arguments from the proof of Theorem \ref{hardy}
one easily shows that there exists $M>0$ such that, for any $l$,
$$
\bigg\|f(R_k z) - \sum_{j=1}^{N_k} \bigg(\int_{I_{k,j}} f(\zeta) dm(\zeta) \bigg)
\frac{1}{1-R_k \bar \zeta_{k,l,j} z}\bigg\|_{H^1} \le \frac{\|f\|_{H^1}}{4}
$$ 
as soon as $N_k(1-R_k) >M$. Thus, we have $\|f(R_k z) - S_k(z)\|_{H^1} \le \|f\|_{H^1}/4$.

We need to estimate the norms of intermediate sums in $S_k$. 
Note that the outer summation goes over $l$. 
It follows from \eqref{dodr} that for $1\le \tilde M \le M_k$
$$
\bigg\|\frac{1}{M_k} \sum_{l=1}^{\tilde M} \sum_{j=1}^{N_k} \bigg(\int_{I_{k,j}} f(\zeta) dm(\zeta) \bigg)
\frac{1}{1-R_k \bar \zeta_{k,l,j} z} -  \frac{1}{M_k} \sum_{l=1}^{\tilde M} \int_\T \frac{f(\zeta)}{1-R_k \bar \zeta z} 
dm(\zeta) \bigg\|_{H^1} \le \frac{\|f\|_{H^1}}{4},
$$
while the second sum inside the norm is simply $\frac{\tilde M}{M_k} f(R_k z)$. Finally, we need to estimate, for some fixed 
$\tilde M$ and $1\le \tilde N\le N_k$, the norm
$$ 
\bigg\|\frac{1}{M_k} \sum_{j=1}^{\tilde N} \bigg(\int_{I_{k,j}} f(\zeta) dm(\zeta) \bigg)
\frac{1}{1-R_k \bar \zeta_{k, \tilde M, j} z} \bigg\|_{H^1} 
$$
or, equivalently, the norm 
$$ 
\bigg\|\frac{1}{M_k} \int_I \frac{f(\zeta)}{1-R_k \bar \zeta z} 
dm(\zeta) \bigg\|_{H^1},
$$
where $I = \cup_{j=1}^{\tilde N} I_{k,j}$. Since $\int_\T |1-\rho \zeta|^{-1} dm(\zeta) 
\lesssim \log \frac{1}{1-\rho}$, $1/2 \le\rho <1$, we have
$$
\frac{1}{M_k} \int_\T \int_I \frac{|f(\zeta)|}{|1-R_k  \bar \zeta z|} 
dm(\zeta) \, dm(z) \lesssim \frac{1}{M_k} \log \frac{1}{1-R_k} \|f\|_{H^1} \lesssim \|f\|_{H^1}
$$
by the hypothesis on $M_k$.

In the case $f\in A(\D)$ the estimate
$$
\bigg\|\frac{1}{M_k} \int_I \frac{f(\zeta)}{1-R_k \bar \zeta z} 
dm(\zeta) \bigg\|_{A(\D)} \lesssim \frac{1}{M_k} \log \frac{1}{1-R_k} \|f\|_{A(\D)} \lesssim \|f\|_{A(\D)}
$$ 
is immediate.

The rest of the proof is identical to the proof of Theorem \ref{hardy}. Let $X$ be one of the spaces $H^1$ or $A(\D)$.
For a fixed $f$ we choose $R_{k_1}$ so that $\|f(z) - f(R_{k_1} z)\|_X \le \|f\|_{X}/4$. Then 
$\|f-S_{k_1}\|_{X} \le \|f\|_{X}/2$.
Applying the procedure to $f_1 = f - S_{k_1}$ we choose  $R_{k_2}$, etc.
\end{proof} 
\bigskip


\section{Proof of Theorem \ref{wei}}

We start with the following integral representation of functions in $\mathscr{H}_\beta$.
In what follows for $f (z)= \sum\limits_{n=0}^{\infty} a_n z^n \in \mathscr{H}_\beta$ and $\rho \in (0,1)$ 
we put
$$
F_\rho(z)= \sum\limits_{n=0}^{\infty} a_n\beta_n \rho^n z^n.
$$
Note that $F_\rho$ is analytic in $\{|z| <\rho^{-1}\}$ 
and
\begin{equation}
\label{zog}
\int_\T |F_\rho(\zeta)|^2 dm(\zeta) = \sum_{n=0}^\infty |a_n|^2 \beta_n^2 \rho^{2n} \le \omega_1(\rho)\|f\|_\beta^2,
\end{equation}
where 
$$
\omega_1(\rho)=\sup_{n\ge 0} \rho^{2n} \beta_n.  
$$

\begin{lemma}
\label{inte}
Let $f(z)= \sum\limits_{n=0}^{\infty} a_n z^n \in \mathscr{H}_\beta$.
Then, for any $0<r<R <1$ and $z\in \D$,
$$
f(rz)=  \int\limits_{\T} F_{r/R} (\zeta) \overline{K^\beta_z(R\zeta)} dm(\zeta).
$$
\end{lemma}

\begin{proof} By direct computations
$$
\int\limits_{\T} F_{r/R} (\zeta) \overline{K^\beta_z(R\zeta)} dm(\zeta) =
\int\limits_{\T} \bigg(\sum\limits_{n=0}^{\infty}
 a_n\beta_n\frac{r^n}{R^n} \zeta^n\bigg)
 \bigg(\sum\limits_{n=0}^{\infty}\frac{R^n z^n }{\beta_n} \bar\zeta^n\bigg) dm(\zeta)
=  \sum\limits_{n=0}^{\infty}  a_n r^n z^n.
$$  
\end{proof}

We will introduce two more characteristics of the sequence $\beta$. For $\rho\in (0,1)$ put
$$
\om_2(\rho) = \sum_{n\ge 1} \frac{n^2 \rho^{2n}}{\beta_n}, \qquad \om_3(\rho) = \sum_{n\ge 0} \frac{\rho^{2n}}{\beta_n}.
$$
Note that $\om_3(\rho)$ is the square of the norm of the reproducing kernel $K^\beta_\rho$ in $\hb$, while $\om_2(\rho)$
is essentially the squared norm of its derivative. 

The key idea of a construction of a representing system is similar to the case of $H^1$.
Assume that for any $k\in \N$ there are fixed $R_k \in(0,1)$ and $N_k, M_k \in \N$
and a collection of radii $R_{k,l} \in (0,1)$, $l=1, \dots, M_k$, such that $R_{k,1} <R_{k,2} \dots 
< R_{k, M_k} =  R_k$. Consider the set of points
$$
\Lambda = \{w_{k, l, j} :\ 1\le l \le M_k, 1\le j\le N_k, k\in\N\},  \qquad  w_{k, l, j} = R_{k,l} \exp \Big(\frac{2\pi i j}{N_k}\Big).
$$

\begin{theorem}
\label{gen}
Assume that $R_{k, 1} \to 1$ and 
\begin{equation}
\label{tout}
\om_1(R_k) \om_2(R_k) = o(N_k^2), \qquad
\om_1(R_k) \om_3(R_k) = O(M_k^2)
\end{equation}
as $k\to\infty$. Then $\{K_\lambda^\beta\}_{\lambda\in\Lambda}$ is a representing system in $\hb$.
\end{theorem}

\begin{proof}
By Lemma \ref{inte} we have for any $r<R_{k,1}$
$$
f(rz) = \frac{1}{M_k} \sum_{l=1}^{M_k} \int_ \T F_{r/R_{k,l}} (\zeta) \overline{K^\beta_z(R_{k,l}\zeta)} dm(\zeta).
$$
The idea of the proof is to discretize this integral replacing it by 
$$
S_k(z) = \frac{1}{M_k} \sum_{l=1}^{M_k} \sum_{j=1}^{N_k} 
\bigg(\int_{I_{k,j}} F_{r/R_{k,l}} (\zeta) dm(\zeta)\bigg)    K^\beta_{w_{k,l,j}}(z),
$$
where $I_{k,j} = [\exp(\frac{(2 j -1)\pi i}{N_k}), \exp(\frac{(2 j +1)\pi i}{N_k})]$. 

We consider in detail one step of approximation. For the moment we assume that $k$ is fixed 
and omit it, i.e., we write $R, R_l, w_{l,j}, M, N$ in place of  $R_k, R_{k,l}, w_{k,l,j}, M_k, N_k$. 
Recall that $R_1<\dots < R_M = R$. Assume that $r < R_1^2$.

Note that $\overline{K^\beta_z(R_l \zeta)} = K^\beta_{R_l \zeta} (z)$. Then we have 
$$
f(rz) - S(z)  =  \frac{1}{M} \sum_{l=1}^{M} \sum_{j=1}^{N} 
\int_{I_j}  F_{r/R_{l}} (\zeta) \big(K^\beta_{R_l \zeta} (z) - K^\beta_{w_{l,j}}(z)\big)
dm(\zeta) = \sum_{n=1}^\infty c_n \frac{z^n}{\beta_n},
$$
where 
$$
c_n =  \frac{1}{M} \sum_{l=1}^{M} \sum_{j=1}^{N} 
\int_{I_j}  \big( (R_l\bar \zeta)^n - \overline{w}_{l,j}^n \big) F_{r/R_{l}}(\zeta) dm(\zeta).
$$
Since $w_{l,j} = R_l \zeta_j$, $\zeta_j = e^\frac{2\pi i j}{N}  \in I_j$, we have for $ \zeta \in I_j$
$$
|(R_l \zeta)^n - w_{l,j}^n| =R_l^n \cdot |\zeta - \zeta_j| \cdot  \Big| \sum_{s=0}^{n-1} \zeta^s \zeta_j^{n-1-s}\Big|
\lesssim \frac{n R_l^n}{N} \le \frac{n R^n}{N},
$$
whence
$$
|c_n| \lesssim \frac{nR^n}{M N} \sum_{l=1}^{M}  \int_\T |F_{r/R_{l}}(\zeta)| dm(\zeta).
$$
Recall that $r\le R_1^2$ and so $r/R_l \le R$.
It follows from \eqref{zog} that 
\begin{equation}
\label{bat}
\int_\T |F_{r/R_{l}}(\zeta)| dm(\zeta) \le (\om_1(R))^{1/2} \|f\|_\beta.
\end{equation}
We conclude that
$$
|c_n| \lesssim \frac{n  R^n (\om_1(R))^{1/2}}{N} \|f\|_\beta.
$$
Hence, 
$$
\|f_r - S\|_\beta^2  = \sum_{n=1}^\infty \frac{|c_n|^2}{\beta_n} \lesssim 
\frac{\om_1(R)   \|f\|^2_\beta}{N^2} \sum_{n=1}^\infty \frac{n^2 R^{2n}}{\beta_n} = 
\frac{\om_1(R)  \om_2(R)}{N^2} \|f\|^2_\beta.
$$
By the first condition in \eqref{tout}, taking $R=R_k$ and $N=N_k$ with a large $k$, one can make this norm as small as we wish. 

Let us show that any intermediate partial sum is uniformly bounded by the norm of $f$. 
We need to estimate the norms of the sums 
$$
T_{\tilde M}(z) = \frac{1}{M} \sum_{l=1}^{\tilde M} \sum_{j=1}^{N} 
\bigg( \int_{I_j}  F_{r/R_{l}} (\zeta) dm(\zeta) \bigg) K^\beta_{w_{l,j}}(z)
$$
and 
$$
T_{\tilde M, \tilde N}(z) = \frac{1}{M} \sum_{j=1}^{\tilde N} 
\bigg( \int_{I_j}  F_{r/R_{\tilde M}} (\zeta) dm(\zeta) \bigg) K^\beta_{w_{\tilde M, j}}(z),
$$
where $1\le \tilde M\le M$ and $1\le \tilde N \le N$. 

First let us consider the sums over several complete circles.
It follows from the above estimates of the difference between the integral and its discretization 
and from Lemma \ref{inte} that for any $1\le \tilde M \le M$ one has
$$
\bigg\| T_{\tilde M}(z)  -
\frac{1}{M} \sum_{l=1}^{\tilde M} \int_ \T  F_{r/R_{l}} (\zeta) \overline{K^\beta_z(R_{l}\zeta)} dm(\zeta) 
\bigg\|_\beta^2
= \bigg\| T_{\tilde M}(z) - \frac{\tilde M}{M} f(rz) \bigg\|_\beta^2 \lesssim  \frac{\om_1(R)  \om_2(R)}{N^2} \|f\|^2_\beta.
$$
Hence, $\| T_{\tilde M} \|_\beta \lesssim \|f\|_\beta$.
It remains to estimate the norm of the sum $T_{\tilde M, \tilde N}(z)$ over some incomplete circle.
Again, by the above estimates, we have
$$
\bigg\|  T_{\tilde M, \tilde N}(z) -  
\frac{1}{M} \int_ I  F_{r/R_{\tilde M}} (\zeta) \overline{K^\beta_z(R_{\tilde M}\zeta)} dm(\zeta) 
 \bigg\|_\beta^2 \lesssim  \frac{\om_1(R)  \om_2(R)}{N^2} \|f\|^2_\beta,
$$
where $I = \cup_{j=1}^{\tilde N} I_j$. Using the expansion of the kernel function we get
$$
\int_ I  F_{r/R_{\tilde M}} (\zeta) \overline{K^\beta_z(R_{\tilde M}\zeta)} dm(\zeta) 
= \sum_{n=0}^\infty  d_n \frac{R_{\tilde M}^n}{\beta_n} z^n, 
$$
where 
$$
d_n = \int_ I  F_{r/R_{\tilde M}} (\zeta) \bar \zeta^n dm(\zeta).
$$
Making use of \eqref{bat} and the fact that $r/R_{\tilde M} \le R$ we get
$$
\bigg\| \frac{1}{M}\int_ I  F_{r/R_{\tilde M}} (\zeta) \overline{K^\beta_z(R_{\tilde M}\zeta)} dm(\zeta) 
\bigg\|_\beta^2  \le \frac{\om_1(R) \|f\|^2_\beta}{M^2} \sum_{n=0}^\infty \frac{R^{2n}}{\beta_n} = 
\frac{ \om_1(R)\om_3(R)}{M^2}  \|f\|^2_\beta \lesssim   \|f\|^2_\beta
$$
by the second condition in \eqref{tout}.

The rest of the proof is analogous to the proof of Theorem \ref{hardy}.
\end{proof}

\begin{remark}
{\rm Representing systems satisfying the condition \eqref{tout} of Theorem \ref{gen} 
are, apparently, more dense than necessary and in special cases these conditions can be 
substantially relaxed. Our goal was to give a qualitative answer to the question about existence 
of representing systems. It is an interesting problem for further research to find optimal density conditions. }
\end{remark}
\bigskip

\section{Open questions}
\label{open}

Representing systems of reproducing kernels in spaces of analytic functions in the disk are far from being well understood. 
While it does not seem reasonable to expect a complete description of representing systems of reproducing kernels
even in $H^2$ setting, a natural question is how small (in some sense) a representing system can be. 
Of course, the smaller the system is, the sharper is the result. One way to measure the size of the system is to introduce a density. 
E.g., for $\Lambda\subset \D$, put
$$
D_+(\Lambda) = \limsup_{r\to 1-}\, (1-r)\cdot \#(\Lambda \cap \{|z| <r\}),
$$
where $\#E$ denotes the cardinality of $E$. In all known examples of representing systems of the Cauchy kernels in $H^p$ 
one has $D_+(\Lambda) >0$.
\medskip
\\
{\bf Question 1.} {\it Do there exist representing systems $\mathcal{K}(\Lambda)$ of the Cauchy kernels in $H^2$
\textup(or $H^p$, $1<p<\infty$\textup), such that $D_+(\Lambda) =0$? }
\medskip

We find it plausible that the answer is ``no'' and so the case when $n_k(1-r_k)$ behaves like a constant is optimal for systems
of the form \eqref{loop}. Of course, one can ask similar questions about representing systems of reproducing kernels in 
general spaces $\hb$ considering
appropriate densities.

Also, in all known examples the points of $\Lambda$ accumulate (and, moreover, nontangentially)
to each point of the unit circle. On the other hand, 
there exist complete systems of the Cauchy kernels which accumulate to a single point on the boundary. 
\medskip
\\
{\bf Question 2.} {\it Do there exist representing systems $\mathcal{K}(\Lambda)$ of the Cauchy kernels in $H^2$
such that the closure ${\rm Clos}\, \Lambda$ does not contain $\T$, i.e., omits some open arc? }
\medskip

As we have seen in Theorem \ref{hardy1}, one can construct representing systems
of the Cauchy kernels in $H^1$ or in $A(\D)$ if one takes somewhat (logarithmically) denser sets, than for $H^p$, $p>1$. 
The question about sharp density remains open. 
\medskip 
\\
{\bf Question 3.} {\it Do the sequences \eqref{loop} with $n_k(1-r_k) \ge M >0$ generate representing systems in $H^1$
or in $A(\D)$? If not, then what is the correct optimal density?}
\medskip

In an interesting paper \cite{cs} (see also \cite{cms}) J.\,A.~Cima and M.~Stessin studied the problem of the constructive recovery of 
a function in a Banach space of analytic functions from its values on a uniqueness set. For a class of spaces in the disk they 
constructed a sequence of approximating functions which are finite sums of reproducing kernels. In particular, 
such a recovery is possible in $H^p$ with $2\le p <\infty$ for {\it any} uniqueness (i.e., non-Blaschke) 
set $\Lambda = \{\lambda_n\}$ and for any
$p\ge 1$ under some additional density conditions on the set (see \cite[Theorems 2 and 4]{cs}). The approximants are finite linear combinations of the Cauchy kernels. Applying the approximation method iteratively (as in the proof of Theorem \ref{hardy}) one can, apparently, construct, for any $f\in H^p$, a series of the form $\sum_n c_n k_{\lambda_n}$ such that
some subsequence of its partial sums converges to $f$. These results seem to be essentially different from our setting 
since for a representing system the whole sequence of partial sums must converge in the norm. This need not be true
unless the set $\Lambda$ has some additional symmetry (a simple example of a uniqueness set which does not generate
a representing system of the Cauchy kernels can be found in \cite[Theorem 3.3]{fkl}). It seems however to be an interesting question, 
whether for a representing system $\mathcal{K}(\Lambda) $ of the Cauchy kernels one can find an explicit expression
of the coefficients in an expansion of $f$ in terms of the values of $f$ on $\Lambda$.
\bigskip
\\
{\bf Acknowledgement.} The authors are grateful to 
Nikolaos Chalmoukis for useful discussions, to Raymond Mortini for attracting their attention to the paper \cite{cs}
and to the referee for numerous helpful remarks.

\bigskip

\noindent
{\bf Statements and Declarations}
\medskip
\\
{\bf Funding:} 
The work is supported by Ministry of Science and Higher Education of the Russian Federation (agreement No 075-15-2021-602) and by Theoretical Physics and Mathematics Advancement Foundation ``BASIS''.
\medskip
\\
{\bf Competing interests:} The authors have no relevant financial or non-financial interests to disclose.
\medskip
\\
{\bf Data availability:} The article does not contain any data for analysis.

\bigskip


\begin{thebibliography}{BRSHZE}


\bibitem{bdk} A. Borichev, R. Dhuez, K. Kellay, Sampling and interpolation in large Bergman and Fock spaces,
{\it J. Funct. Anal.} {\bf  242} (2007), 2, 563--606.

\bibitem{cms}
J.\,A. Cima, T.\,H. MacGregor, M.\,I. Stessin, Recapturing functions in $H^p$ spaces, 
{\it Indiana Univ. Math. J.} {\bf 43} (1994), 1, 205--220.

\bibitem{cs}
J.\,A. Cima, M. Stessin, On the recovery of analytic functions, {\it Canad. J. Math.} {\bf 48} (1996), 2, 288--301.

\bibitem{dur} 
P.\,L. Duren, {\it Theory of $H^p$ Spaces}, Academic Press, New York, 1970. 

\bibitem{fkl} E. Fricain, L.\,H. Khoi, P. Lef\`evre, 
Representing systems generated by reproducing kernels, {\it Indag. Math.} {\bf 29} (2018), 3, 860--872.

\bibitem{is} K.\,P. Isaev, Representing exponential systems in spaces of analytical functions, 
Itogi Nauki Tekh. Ser. Sovrem. Mat. Prilozh. Temat. Obz., 161, Complex analysis. Entire functions and their applications, 3--64, 2019;
English transl.: {\it J. Math. Sci.} {\bf 257} (2021), 2, 143--205. 

\bibitem{kor} Yu.\,F. Korobeinik, Representative systems, {\it Uspekhi Mat. Nauk} {\bf 36} (1981), 1,  73--126;
English. transl.:  Russian Math. Surveys {\bf 36} (1981), 1, 75--137.

\bibitem{seip} K. Seip, Beurling type density theorems in the unit disk, {\it Invent. Math.} {\bf 113} (1993), 1, 21--39.

\bibitem{seip1} K. Seip, Interpolation and sampling in small Bergman spaces, {\it Collect. Math.} {\bf 64} (2013), 61--72.

\bibitem{st1} K.\,S. Speransky, P.\,A. Terekhin, 
A representing system generated by the Szeg\"o kernel for the Hardy space,
{\it Indag. Math.} {\bf 29} (2018), 5, 1318--1325.

\bibitem{st2} K.\,S. Speransky, P.\,A. Terekhin, 
On existence of frames based on the Szeg\"o kernel in the Hardy space, {\it Izv. VUZ. Matem.} {\bf 2019}, 2,  57--68;
English transl.: {\it Russian Math. \textup(Izv. VUZ\textup)} {\bf 63} (2019), 2, 51--61.

\bibitem{ter1}
P.\,A. Terekhin, Banach frames in the affine synthesis problem, {\it Sb. Math.} {\bf 200} (2009), 9, 1383--1402. 

\bibitem{ter2}
P.\,A. Terekhin, Frames in Banach spaces, {\it Funct. Anal. Appl.} {\bf 44} (2010), 3, 199--208.
\end{thebibliography}
\end{document}